\documentclass[preprint]{elsarticle}

\makeatletter
\def\ps@pprintTitle{%
	\let\@oddhead\@empty
	\let\@evenhead\@empty
 
	\def\@oddfoot{\footnotesize\itshape
		{} \hfill\today}%
	\let\@evenfoot\@oddfoot
}
\makeatother

\usepackage{latexsym}
\usepackage{lineno}
\usepackage{hyperref}
\usepackage{indentfirst}
\usepackage{amsxtra}
\usepackage{amssymb}
\usepackage{amsthm}
\usepackage{natbib}
\usepackage{amsmath}
\usepackage{tikz}
\usepackage{amscd}
\usepackage{MnSymbol}

\usepackage[capitalise]{cleveref}
\newtheorem{theor}{Theorem}
\newtheorem{prop}[theor]{Proposition}
\newtheorem{cor}[theor]{Corollary}
\newtheorem{lemma}[theor]{Lemma}

\theoremstyle{definition} 
\newtheorem{defin}{Definition}
\newtheorem{rem}{Remark}

\newtheorem{ex}[theor]{Example}
\newtheorem{exs}[theor]{Examples}

\DeclareMathOperator{\Ret}{Ret}

\begin{document}

\begin{frontmatter}
	\title{On commutative set-theoretic solutions of the Pentagon Equation 
 }
	\tnotetext[mytitlenote]{The author is a member of GNSAGA (INdAM).}
	\author{Marco CASTELLI}
	\ead{marco.castelli@unisalento.it - marcolmc88@gmail.com}
	%

\begin{abstract}
We extend the so-called \emph{retract relation} given in \cite{colazzo2020set} for involutive set-theoretic solutions of the Pentagon Equation and we introduce the notion of \emph{associated permutation group} to study the family of the commutative non-degenerate ones. Moreover, we develop a machinery to construct all these solutions and we use it to give a quite explicit classification of the irretractable ones. Finally, non-degenerate solutions on left-zero semigroup are studied in detail, with an emphasis on the ones with cyclic associated permutation group and on the ones having small size.
\end{abstract}
\begin{keyword}
\texttt{Set theoretic solution, Pentagon equation }
\MSC[2020] 81R50 20M99
\end{keyword}

\end{frontmatter}

\section*{Introduction}

A pair $(\mathcal{S},V)$ is said to be a \emph{solution of the Pentagon Equation} if $V$ is a vector space, $\mathcal{S}$ is a linear map and the equality
$$\mathcal{S}_{23}\mathcal{S}_{13}\mathcal{S}_{12}=\mathcal{S}_{12}\mathcal{S}_{23} $$
holds, where $ \mathcal{S}_{ij}$ is the map from $V\otimes V\otimes V$ to itself that acts as $\mathcal{S}$ on the $(i,j)$-factor and as the identity on the remaining factor. Solutions of the Pentagon Equation appeared in several context, such as Hilbert spaces, quantum groups and Hopf algebras; here, we give only some references \cite{baaj1993unitaires,kashaev2011fully,kawamura2010pentagon,militaru2004heisenberg,moore1989classical,woronowicz1996multiplicative}.\\
In recent years, several people studied this equation using teqniques and ideas that are similar to the ones used for the \emph{Quantum Yang Baxter equation}, another famous equation that comes from the field of Mathematical Physics. In the context of the Quantum Yang Baxter equation, several papers, following the approach suggested by Drinfield (see \cite{drinfeld1992some}), focus on the classification-problem of the simplified case of \emph{set-theoretic} solutions. Bijective non-degenerate set-theoretic solutions received a lot of attention and several authors developed various teqniques to find and classify them (see for example \cite{etingof1998set,gateva2004combinatorial,guarnieri2017skew,rump2005decomposition}). In this context, two useful tools reveal their importance: the \emph{retract relation}, introduced in \cite{etingof1998set} for involutive solutions and considered in \cite{JeP18} for bijective solutions, and the \emph{associated permutation group}. The retract relation, which is an equivalence relation defined on the underlying set $X$ of a solution, allows to construct new solutions of smaller cardinality and provides a way to study an arbitrary solution focusing on the so-called \emph{retraction-process}. The associated permutation group, which is a permutation group generated by two standard maps provided by a non-degenerate solution, allow to give structural informations on solutions and provides a natural approach to find and classify bijective non-degenerate solutions: indeed, given an arbitrary group $G$, one can try to construct and classify all the bijective non-degenerate solutions with associated permutation group isomorphic to $G$. Even if we are far from a complete classification of set-theoretic solutions of the Yang-Baxter equation, several families of solutions are completely understood (see for example \cite{castelli2023studying,JePiZa20x,jedlivcka2021cocyclic,rump2022class}).\\
As we said before, some authors introduced in the context of the Pentagon Equation some ideas and tools used to study Yang-Baxter Equation. In that regard, in 1998 Kashaev and Sergeev \cite{kashaev1998pentagon} began the study of the set-theoretical version of this equation. Specifically, a pair $(S,s)$ is said to be a \textit{set-theoretic solution of the Pentagon Equation} if the equality
 $$ s_{23}s_{13}s_{12} = s_{12}s_{23}$$
 holds, where $s_{12} = s\times id$, $s_{23} = id\times s$ and $s_{13} = (\tau \times id)(id \times s)(\tau\times id)$ are mappings from $S^3$ to itself and $\tau$ is the flip map. Similarly to what happens in the context of Quantum Yang-Baxter Equation, every set-theoretic solution $(S,s)$ induces a solution of the Pentagon Equation $(V,\mathcal{S})$, where $V$ is a vector space having $S$ as a basis. Following the terminology given in \cite{catino2020set1} and writing the the map $s$ as $s(x,y):=(x\cdotp y,\theta_x(y))$ for all $x,y\in S$, several authors investigated these solutions by several algebraic structures associated to the operation $\cdotp$ or to the maps $\theta_x$. In \cite{kashaev1998pentagon}, Kashaev and Sergeev completely describe the invertible solutions under the assumption that $(S,\cdotp)$ is a group. This description has been extended by Catino, Mazzotta and Miccoli in \cite{catino2020set1}, where the invertibility of solutions is dropped. Moreover, they completely classified all the solutions $(S,s)$ for which the operation $x*y:=\theta_x(y)$ makes $S$ into a group structure. Other remarkable results has been obtained in \cite{mazzotta2023set}, where the authors studied solutions on Clifford semigroup.\\
By a different point of view, other papers focus on the map $s$. In \cite{catino2020set}, maps $s$ which makes $(S,s)$ into a set-theoretic solutions of the Pentagon Equation that is also a set-theoretic solution of the Yang-Baxter equation have been characterised. In \cite{mazzotta2023idempotent} Mazzotta studied \emph{idempotent} solutions, i.e. solutions $(S,s)$ for which $s^2=s$, while in \cite{colazzo2020set} Colazzo, Jespers and Kubat provide a complete description of \emph{involutive} solutions, i.e. solutions $(S,s)$ for which $s^2=id_{S\times S}$. The main tool used to describe involutive solutions is the retract relation, which is very similar to the one introduced by Etingof, Schedler and Soloviev to study involutive solution of the Yang-Baxter equation \cite{etingof1998set}. Further results involving set-theoretic solutions of the Pentagon Equation and other areas of Mathematics can be found in \cite{baaj2003unitaires,jiang2005set,kashaev2007symmetrically}. \\
In this paper, we study commutative non-degenerate set-theoretic solutions of the Pentagon Equation. By results contained in \cite[Section 3]{colazzo2020set}, this family of solutions includes the involutive ones, even if we will explicitly provide examples that are different from the ones given in \cite{colazzo2020set}. To our goal, we will introduce two key-ingredients: the retract relation, which is consistent with the one introduced in \cite{colazzo2020set} for involutive set-theoretic solutions, and the \emph{associated permutation group}, which is similar to the one introduced in \cite{etingof1998set} for involutive set-theoretic solutions of the Yang-Baxter equation. As a first main result of the paper, we will show that the action of the associated permutation group is always semi-regular. This fact will be useful in two directions: in the first one, it provides a method to construct all the commutative non-degenerate set-theoretic solutions; in the second one, it allows to detect and characterise easily all the irretractable solutions, which will be studied in detail in Section $4$. In the main result of this section, we will show that the isomorphism class of the associated permutation group is all we need to construct all the finite irretractable solutions. As a corollaries, we will extend some results contained in \cite{colazzo2020set} in the context of involutive solutions.\\
Following \cite[Problem 7]{mazzotta2023survey}, in the last part of the paper we will focus on solutions on left-zero semigroups. Here, we will develop an extension-tool useful to construct all the non-degenerate ones. A family of solutions that includes the one exhibited in \cite[Proposition 5.1]{colazzo2020set} will be provided. Moreover, the non-degenerate ones having cyclic associated permutation group will be classified. We will close the paper providing some classification-results on solutions having small size.

\section{Basic definitions and results}

 Let $S$ be a non-empty set and $s$ a map from $S\times S$ to itself. Then, the pair $(S,s)$ is said to be a \textit{set-theoretic solution of the Pentagon Equation} if the equality
 $$ s_{23}s_{13}s_{12} = s_{12}s_{23}$$
 follows, where $s_{12} := s\times id$, $s_{23} := id\times s$ and $s_{13} := (\tau \times id)(id \times s)(\tau\times id)$ are mappings from $S^3$ to itself and $\tau$ is the map from $S\times S$ to itself given by $\tau(x,y):=(y,x)$ for all $x,y\in S$. From now on, the pair $(S,s)$ will be simply called \emph{solution}.\\
 We say that $(S,s)$ is \emph{finite} if $S$ is finite as a set and \emph{bijective} if $s$ is a bijective map. Moreover, we will refer to the \emph{order} of a bijective solution $(S,s)$ as the order of the map $s$. Writing $s(x,y)=(x\cdotp y,\theta_x(y))$, by a routine computation, one can show that a pair $(S,s)$ is a solution if and only if the following equalities
 \begin{equation}\label{eq1}
     (x\cdotp y)\cdotp z=x\cdotp (y\cdotp z)
     \end{equation}
     \begin{equation}\label{eq2}
              \theta_x(y)\cdotp \theta_{x\cdotp y}(z)=\theta_x(y\cdotp z)\\
     \end{equation}
   \begin{equation}\label{eq3}
     \theta_{\theta_x(y)}\theta_{x\cdotp y}=\theta_y
 \end{equation}
 hold for all $x,y,z\in S$. If $(S,s)$ is a solution, by \cref{eq1} the algebraic structure $(S,\cdotp)$ must be a semigroup that we will call the \emph{underlying} semigroup and we will say that $(S,s)$ is a solution on the semigroup $(S,\cdotp)$. 

 \begin{exs}\label{primies}
     \begin{itemize}
         \item[a)] If $(S,\cdotp)$ is a semigroup and $f$ an idempotent endomorphism, then the map $s$ from $S\times S$ to itself given by $s(x,y):=(x\cdotp y, f(y))$ give rise to a solution of the Pentagon equation (see \cite{catino2020set1}).
         \item[b)] If $(S,+)$ is an elementary abelian $2$-group, then the map $s$ from $S\times S$ to itself given by $s(x,y):=(x,x+y)$ give rise to a bijective solution of the Pentagon equation (see \cite{colazzo2020set}).
          \item[c)] If $S$ is a a set and $f,g$ are idempotent commuting maps from $S$ to itself, then the map $s:S\times S\longrightarrow S\times S$ given by $s(x,y):=(f(x),g(y))$ give rise to a solution of the Pentagon equation (see \cite{militaru98}).
     \end{itemize}
 \end{exs}

 A special classes of solutions that are considered in literature are the commutative solutions (see \cite{baaj1993unitaires}) and the non-degenerate solutions (recently introduced in \cite{mazzotta2023survey}).

\begin{defin}
   A solution $(S,s)$ is said to be \textit{commutative} if and only if $s_{12}s_{13}=s_{13}s_{12}$.
\end{defin}

\noindent By a simple calculation, we have that $(S,s)$ is a commutative solution on a semigroup $(S,\cdotp)$ if and only if 
$$x\cdotp y\cdotp z=x\cdotp z\cdotp y\qquad and \qquad \theta_{x\cdotp y}=\theta_x$$
for all $x,y,z\in S$. The semigroups satisfying the first equality will be called \emph{weak commutative} semigroups.

\begin{defin}(\cite[Section 1]{mazzotta2023survey})
   A solution $(S,s)$ is said to be \textit{non-degenerate} if $\theta_x$ is bijective for all $x\in S$. 
\end{defin}

Example b) of \cref{primies} provides a family of commutative non-degenerate solution. By results contained in \cite{colazzo2020set}, commutative solutions and non-degenerate solutions contain the family of the involutive ones.

\begin{prop}[Corollaries $3.3$ and $3.4$, \cite{colazzo2020set}]
    Every involutive solution is commutative and non-degenerate.
\end{prop}

\noindent Other examples of commutative or non-degenerate solutions come from solutions on left-zero semigroups, where a semigroup $(S,\cdotp)$ is said to be a \emph{left-zero} semigroup if $x\cdotp y=x$ for all $x,y\in S$.

\begin{prop}
    Let $(S,s)$ be a solution on a left-zero semigroup. Then, $(S,s)$ is a commutative solution. Moreover, $(S,s)$ is bijective if and only if it is non-degenerate.
\end{prop}

\begin{proof}
    Staightforward.
\end{proof}

For a non-degenerate solution $(S,s)$ we introduce a standard permutation group.

\begin{defin}
    Let $(S,s)$ be a non-degenerate solution. Then, the subgroup of $Sym(S)$ generated by the maps $\theta_x$ will be called the \emph{associated permutation group} and will be indicated by $\mathcal{G}(S,s)$.
\end{defin}

\begin{ex}
    If $(S,+)$ is an elementary abelian $2$-group and $(S,s)$ is the solution given by Example b) of \cref{primies}, the permutation group $\mathcal{G}(S,s)$ is isomorphic to $(S,+)$.
\end{ex}

\begin{defin}
    Two solutions $(S,s)$ and $(T,t)$ are said to be \emph{isomorphic} if there exist a bijection $\phi:S\longrightarrow T$ such that $(\phi\times \phi)s=t(\phi\times \phi)$
\end{defin}

By a standard calculation, one can show that if $\phi$ is an isomorphism of two solutions $(S,s)$ and $(T,t)$, then the underlying semigroups of these two solutions $(S,\cdotp)$ and $(T,\star)$ are isomorphic and $\phi$ also is a semigroup isomorphism.\\

We close the section giving the new (but standard) notion of subsolution.

\begin{defin}
    Let $(S,s)$ be a solution on a semigroup $(S,\cdotp)$. A subset $T$ of $S$ is a \emph{subsolution} of $(S,s)$ if $(T,s_{|_{T\times T}})$ again is a solution.
\end{defin}

By a standard calculation, one can show that $T$ is a subsolution of a solution $(S,s)$ if and only if $T$ is a subsemigroup of $(S,\cdotp)$ and $\theta_x(y)\in T$ for all $x,y\in T$. 

\begin{ex}
    Let $(S,s)$ be a finite solution on a left-zero semigroup $(S,\cdotp)$. Then, every union of orbits (respect to the action of $\mathcal{G}(S,s)$) of $S$ is a subsolution of $(S,s)$.
\end{ex}

\section{Commutative non-degenerate solutions and semi-regular groups}

In this section, we focus on some properties relating the structure of the associated permutation group, the underlying semigroups and the maps $\theta_x$. In particular, we show that the action of any associated permutation group is semi-regular and we provide a construction of commutative non-degenerate solutions starting from a weak commutative semigroup $(S,\cdotp)$ and a semi-regular subgroup $\mathcal{G}$ of $Aut(S,\cdotp)$.

\bigskip

At first, we show two technical lemmas. Recall that a permutation group $\mathcal{G}$ acting on a set $S$ is said to be \emph{semi-regular} whenever the equality $g(s)=g'(s)$, with $g,g'\in \mathcal{G}$ and $s\in S$, implies $g=g'$ (or, equivalently, every element $s\in S$ has trivial stabilizer, see \cite[pg. 31]{robinson2012course}). 

\begin{lemma}\label{lemimp1}
     Let $(S,s)$ be a commutative non-degenerate solution. Then, for every orbit $S_i$ (respect to the action of $\mathcal{G}(S,s)$), there exist $z\in S_i$ such that $\theta_z=id_S$.
\end{lemma}

\begin{proof}
    If $x$ is an arbitrary element of $S_i$ and we set $z:=\theta_x(x)$, we obtain that $z\in S_i$ and by \cref{eq3} we have $\theta_z=\theta_x\theta_x^{-1}=id_S$, therefore the thesis follows.
\end{proof}

\begin{lemma}\label{lemimp11}
     Let $(S,s)$ be a finite commutative non-degenerate solution and $z$ an element of $S$. Then, $\theta_{g(z)}=\theta_z g^{-1}$ for all $g\in \mathcal{G}(S,s)$.
\end{lemma}

\begin{proof}
    Since $S$ is finite, every element $g\in \mathcal{G}(S,s)$ can be written as $\theta_{x_1}...\theta_{x_n}$ for some suitable $n\in \mathbb{N}$ and $ x_1,...,x_n\in S$. Therefore, we show the claim by induction on $n$. If $n=1$, being $(S,s)$ commutative, the thesis follows by \cref{eq3}. Now, suppose that $\theta_{\theta_{x_1}...\theta_{x_{n-1}} (z)}=\theta_z (\theta_{x_1}...\theta_{x_{n-1}})^{-1}$ for a fixed $n\in \mathbb{N}$ and every $x_1,...,x_{n-1},z\in S$. Then by inductive hypothesis we have
    $$\theta_{\theta_{x_1}...\theta_{x_{n}} (z)}=\theta_{\theta_{x_n}(z)} (\theta_{x_1}...\theta_{x_{n-1}})^{-1}$$
    for all $x_1,...,x_n,z\in S$ and applying again \cref{eq3} we obtain     
    $$\theta_{\theta_{x_n}(z)}=\theta_{z} \theta_{x_n}^{-1}(\theta_{x_1}...\theta_{x_{n-1}})^{-1}=\theta_{z} (\theta_{x_1}...\theta_{x_{n}})^{-1}$$
    therefore the thesis follows.
\end{proof}

As a first result of the section, we show that the action of the permutation group $\mathcal{G}(S,s)$ of an arbitrary finite commutative non-degenerate solution is semi-regular.

\begin{theor}\label{teo1}
      Let $(S,s)$ be a finite commutative non-degenerate solution. The action of the group $\mathcal{G}(S,s)$ on $S$ is semi-regular.
\end{theor}

\begin{proof}
  If $g,g'\in \mathcal{G}$ and $z\in S$ are such that $g(z)=g'(z)$, then we have that $\theta_{g(z)}=\theta_{g'(z)}$ and by \cref{lemimp11} the equality $\theta_z g^{-1}=\theta_z g'^{-1}$ follows, hence $g=g'$ and the thesis follows.
\end{proof}

\begin{rem}
   In \cite[Lemma 2.4]{colazzo2020set} it was shown that if $(S,s)$ is a bijective solution on a left-zero semigroup, then $(S,s)$ is non-degenerate and all the $\theta_x$ that are different from the identity map are fixed-point-free permutations. Since all the bijective solutions on left-zero semigroups are commutative, \cref{teo1} give an extension of the second part of \cite[Lemma 2.4]{colazzo2020set} in the finite case. 
\end{rem}

Following \cite[Chapter 1]{wielandt2014finite}, if $G$ is a group acting on a set $S$, a subset $S_1$ of $S$ is called a \emph{block} if for each $g\in G$ the set $g(S_1)$ coincides with $S_1$ or has no elements common with $S_1$. A partition $\{S_1,...,S_r \}$ of $S$ is said to be a \emph{system of blocks} if each $S_i$ is a block. The following result provides a method to construct new commutative non-degenerate solutions on a weak commutative semigroup.

\begin{theor}\label{cosimpo}
      Let $(S,\cdotp)$ be a weak commutative semigroup, $\sim$ a congruence of $S$ such that $S/\sim$ is a left-zero semigroup and $\mathcal{G}$ be a subgroup of $Aut(S,\cdotp)$ such that its action on $S$ is semi-regular. Suppose that $S_1,...,S_r$ are the orbits respect to this action and $\mathcal{T}:=\{ T_1,...,T_m\}$ the partition induced by $\sim$. Moreover, assume that every $S_i$ is a set of representatives of $S/\sim$ and $\mathcal{T}$ is a system of blocks of $\mathcal{G}$. \\
      Now, choose $T$ in $\mathcal{T}$ and let $\theta$ be the map from $S$ to $\mathcal{G}$ given by $\theta_{s}:=id_S$ and $\theta_{g(s)}:=g^{-1}$ for all $s\in T$ and $g\in \mathcal{G}$.
      Then, the map $s:S\times S\longrightarrow S\times S$ given by $s(x,y):=(x\cdotp y,\theta_x(y))$ give rise to a commutative non-degenerate solution.
\end{theor}

\begin{proof}
Being $(S,\cdotp)$ a semigroup, \cref{eq1} follows. Now, since the action of $\mathcal{G}$ on $S$ is semi-regular and $|S_i\cap T_j|=1$ for all $i\in \{1,...,r\}$, $j\in \{1,...,m\}$, a standard calculation shows that the map $\theta$ is well-defined. Moreover, being $S/\sim$ a left-zero semigroup, we have that $x\cdotp y\sim x$ for all $x,y\in S$ and hence $\theta_{x\cdotp y}=\theta_x$ for all $x,y\in X$. This facts, toegheter with $\mathcal{G}\leq Aut(S,\cdotp)$, implies \cref{eq2}. Therefore, to show the thesis, since $\theta_x\in Sym(S)$ for all $x\in S$ and $S$ is a weak commutative semigroup, it is sufficient to prove the equality
$$\theta_{\theta_x(y)}=\theta_y \theta_x^{-1} $$
for all $x,y\in S$. Let $a,b\in S$ and consider $i\in \{1,...,r\}$ and $s\in S$ such that $b\in S_i$ and $s\in S_i\cap T$. Therefore there exist $g\in \mathcal{G}$ such that $g(s)=b$ and hence 


$$\theta_{\theta_a(b)}=\theta_{\theta_a(g(s_i))}=(\theta_ag)^{-1}=g^{-1}\theta_a^{-1}=\theta_b\theta_a^{-1}$$

and since $a $ and $b$ are arbitrary elements of $S$ the thesis follows.
\end{proof}


We apply the previous Theorem on a simple example.

\begin{ex}
    Let $(S,\cdotp)$ be a left group. Then, $S$ can be identified with $E\times G$, where $E$ is a left-zero semigroup and $G$ a group. Let $\mathcal{G}$ be a regular permutation subgroup of $Sym(E)$ and $\sim$ be the congruence on $S$ induced by the equality of the first component. With the same notation of \cref{cosimpo} we have that $\mathcal{T}=\{\{e\}\times G \}_{e\in E}$ and identifying $\mathcal{G}$ with a subgroup of $Aut(S,\cdotp)$ by $g(e,f):=(g(e),f)$ for all $g\in \mathcal{G}$, $e\in E$ and $f\in G$, we obtain that every orbit $S_{i}$ of this action is of the form $E\times \{g_i\}$, with $g_i\in G$. Moreover, we have that $|S_{i}\cap T_j|=1$ for all $i,j$ and $g(\{e\}\times G)=\{g(e)\}\times G$ for all $e\in E$ and $g\in \mathcal{G}$, therefore every $S_{i}$ is a set of representatives of $S/\sim$ for every $i\in \{1,...,r\}$ and $\mathcal{T}$ is a system of blocks of $\mathcal{G}$. Therefore, all the hypothesis of \cref{cosimpo}  hold and hence we can use them to construct a commutative non-degenerate solution on $(S,\cdotp)$.
\end{ex}

We close the section showing that the construction exhibited in \cref{cosimpo} can be greatly simplified if we want to provide solution on left-zero semigroups.

\begin{cor}\label{coroperir}
    Let $(S,\cdotp)$ be a left-zero semigroup, $\mathcal{G}$ be a subgroup of $Sym(S)$ and suppose that the action on $S$ is semi-regular. Moreover, let $S_1,...,S_r$ be the orbits of $S$ respect to this action. For every $i\in \{1,...,r\}$, let $s_i$ be an element of $S_i$ and let $\theta$ be the map from $S$ to $\mathcal{G}$ given by $\theta_{s_i}:=id_S$ and $\theta_{x_i}:=g^{-1}$ for all $x_i\in S_i$ and $i\in \{1,...,r\}$, where $g$ is the unique element of $\mathcal{G}$ such that $g(s_i)=x_i$. Then, the map $s:S\times S\longrightarrow S\times S$ given by $s(x,y):=(x,\theta_x(y))$ give rise to a commutative non-degenerate solution on the semigroup $(S,\cdotp)$.
\end{cor}

\begin{proof}
    Clearly, it is not restrictive supposing that $\mathcal{G}$ is a subgroup of $Aut(S,\cdotp)$. Now, for every $g\in \mathcal{G}$, set $T_g:=\{g(s_i)\}_{s_i\in S_i}$. Clearly, $\mathcal{T}:=\{T_g\}_{g\in \mathcal{G}}$ is a system of blocks for $\mathcal{G}$ and, if we take $\sim$ as the equivalence relation induced by $T$, every $S_i$ is a set of representatives for $\sim$. Therefore, the thesis follows by \cref{cosimpo}.
\end{proof}

\section{Retraction of commutative non-degenerate solutions}

Following \cite[Problem 4]{mazzotta2023survey}, in this section we extend the notion of retraction, introduced in \cite[Section $4$]{colazzo2020set} for involutive solutions, to non-degenerate commutative ones. We will show that some properties proved in  \cite[Section $4$]{colazzo2020set} for involutive solutions also follows in this more general context. As a main application of the retract relation, we will show that the construction provided in \cref{cosimpo} allow to construct \emph{all} commutative non-degenerate solutions.\\

Given a commutative non-degenerate solution $(S,s)$ on a semigroup $(S,\cdotp)$, consider the relation $\sim$, which we will call \textit{retraction}, given by 
$$\forall x,y\in S \quad x\sim y :\Longleftrightarrow \theta_x=\theta_y. $$
Now, we show that $\sim$ induces a solution on the quotient $S/\sim$ (the idea is the same used in \cite{colazzo2020set}, but here we can not use the hypothesis $s^2=id_{S\times S}$). Suppose that $x_1\sim x_2$ and $y_1\sim y_2$. Then, we have that $\theta_{x_1\cdotp y_1}=\theta_{x_1}=\theta_{x_2}=\theta_{x_2\cdotp y_2}$ for all $x_1,x_2,y_1,y_2\in S$, hence $\sim$ is a congruence of the semigroup $(S,\cdotp)$. Moreover, since all the $\theta_x$ are bijective, it follows that  $$\theta_{\theta_{x_1}(y_1)}=\theta_{y_1}\theta^{-1}_{x_1\cdotp y_1}=\theta_{y_1}\theta^{-1}_{x_1}=\theta_{y_2}\theta^{-1}_{x_2}=\theta_{y_2}\theta^{-1}_{x_2\cdotp y_2}=\theta_{\theta_{x_2}(y_2)}$$
and hence $\theta_{x_1}(y_1)\sim \theta_{x_2}(y_2)$.\\
In this way, denoting by $\bar{x}$ the $\sim$-class of an element $x$ of $S$, we showed that the product $\bar{x}\cdotp \bar{y}:=\overline{x\cdotp y}$ and the map $\theta_{\bar{x}}(\bar{y}):=\overline{\theta_x(y)}$ are well-defined. These facts allow to provide a solution $\bar{s}$ on $S/\sim$ by $\bar{s}(\bar{x},\bar{y}):=(\bar{x}\cdotp \bar{y},\theta_{\bar{x}}(\bar{y})) $
for all $\bar{x},\bar{y}\in S/\sim$. We will call this solution \textit{retraction} of $(S,s)$ and we indicate it by $\Ret(S,s)$. By a standard calculation, it follows that $\bar{s}$ again is commutative and  the semigroup $\bar{S}:=S/\sim$ is a left-zero semigroup. Clearly, all the maps $\bar{\theta}_{\bar{x}}$ are surjective, hence $(\bar{S},\bar{s})$ is non-degenerate if $S$ has finite size. In the following, we show that finiteness is not necessary to non-degeneracy. 



\begin{theor}
    Let $(S,s)$ be a commutative non-degenerate solution. Then, $\Ret(S,s)$ is a commutative non-degenerate solution. 
\end{theor}

\begin{proof}
      By commutativity of $(S,s)$, the one of $\Ret(S,s)$ also follows. We only have to show that $\Ret(S,s)$ is non-degenerate. Since all the maps $\theta_x$ are surjective, we obtain that all the maps $\bar{\theta}_{\bar{x}}$ also are surjective. By \cref{lemimp1}, there exist $z\in S$ such that $\theta_z=id_S$, hence we have that $\bar{\theta}_{\bar{z}}=id_{S/\sim}$. Since $\bar{\theta}_{\bar{\theta}_{\bar{x}}(\bar{y})}\bar{\theta}_{\bar{x}}=\bar{\theta}_{\bar{y}}$ for all $\bar{x},\bar{y}\in \bar{S}$ if we set $\bar{y}:=\bar{z}$ we obtain 
      $$\bar{\theta}_{\bar{\theta}_{\bar{x}}(\bar{z})}\bar{\theta}_{\bar{x}}=id_{S/\sim} $$
      hence $\bar{\theta}_{\bar{x}}$ is injective and therefore the thesis follows.
\end{proof}

\begin{defin}
    A commutative non-degenerate solution $(S,s)$ is said to be \textit{irretractable} if $(S,s)=\Ret(S,s)$, otherwise, we will call it \textit{retractable}.
\end{defin}

In the following result we show that all the commutative non-degenerate solutions arise by \cref{cosimpo}. At first, we need a crucial lemma for the rest of the paper.

\begin{lemma}\label{lemgr}
    Let $(S,s)$ a finite commutative non-degenerate solution and $g\in \mathcal{G}(S,s)$. Then, for every orbit $Z$ respect to the action of $\mathcal{G}(S,s) $ there exist a unique $x_z\in Z$ such that $\theta_{x_z}=g$. 
\end{lemma}

\begin{proof}
   Let $Z$ be an orbit of $S$ respect to to the action of $\mathcal{G}(S,s) $. Then, by \cref{lemimp1} there exist $z\in Z$ such that $\theta_z=id_{S}$. By \cref{lemimp11} we have $\theta_{g^{-1}(z)}=g$,  and since $g^{-1}(z)\in Z$ the existence follows. Now, if $x$ and $y$ are two elements of $Z$ such that $\theta_x=\theta_y=g$, if we call $h$ the element of $\mathcal{G}(S,s)$ such that $h(x)=y$ by \cref{lemimp11} we obtain $g=\theta_y=\theta_{h(x)}=\theta_xh^{-1}=gh^{-1}$, therefore $h=id_S$ and hence the thesis follows.
\end{proof}

\begin{cor}\label{invsol}
    Every commutative non-degenerate solution can be constructed by \cref{cosimpo}.
\end{cor}

\begin{proof}
    Let $(S,s)$ be a commutative non-degenerate solution on a semigroup $(S,\cdotp)$. By commutativity of $(S,s)$ it follows that the semigroup $(S,\cdotp)$ is weak commutative. By \cref{eq2} and by commutativity of $(S,s)$, it follows that $\mathcal{G}(S,s)$ is a subgroup of $Aut(S,\cdotp)$ and by \cref{teo1} it acts semi-regularly on $S$. If we take $\sim_r$ as the rectract relation and $\mathcal{T}:=\{T_1,...,T_m\}$ the partition induced by $\sim_r$, by \cref{lemgr} we have that every orbit $S_i$ is a set of representative of $S/\sim_r$ and by \cref{lemimp11} $\mathcal{T}$ is a system of blocks for $\mathcal{G}(S,s)$. Therefore, using the same notation of \cref{cosimpo}, the solution $(S,s)$ can be constructed taking the semigroup $(S,\cdotp )$ and setting $\mathcal{G}:=\mathcal{G}(S,s)$ and $\sim:=\sim_r$.
\end{proof}

\begin{cor}
    Every commutative non-degenerate solution on a left-zero semigroup can be constructed by \cref{coroperir}.
\end{cor}

\begin{proof}
 Using \cref{coroperir} instead of \cref{cosimpo}, the proof is similar to the one of the previous corollary.
\end{proof}

In \cite[Lemma 5.2]{colazzo2020set} Colazzo, Jespers and Kubat showed that all retract classes of an involutive solution have the same cardinality. In the next result, which closes the section, we show that this fact indeed holds for an arbitrary finite commutative non-degenerate solution. Moreover, we show that the associated permutation group allows to determine the size of every class.

\begin{prop}\label{cardretr}
     Let $(S,s)$ be a finite commutative non-degenerate solution. Then, $|\Ret(S,s)|=|\mathcal{G}(S,s)|$ and all the equivalence classes of $S$ respect to the retract relation have cardinality $\frac{|S|}{|\mathcal{G}(S,s)|}$.
\end{prop}

\begin{proof}
    By \cref{lemgr}, we have that $|\Ret(S,s)|$ is equal to $|\mathcal{G}(S,s)| $. By the same lemma, the cardinality of every retract class is equal to the number of orbits of the action of $\mathcal{G}(S,s)$. Since this action is semi-regular, the thesis follows.
\end{proof}

\section{Irretractable solutions}

In this section, we focus on irretractable solutions. In particular, we characterise these solutions by the associated permutation group and we provide a simple method to construct them and to cut the isomorphism classes.

\vspace{2mm}

At first, we extend the result obtained in \cite[Proposition $4.2$]{colazzo2020set} on the retraction-process of involutive solutions to commutative non-degenerate solutions.

\begin{prop}\label{retrirr}
    Let $(S,s)$ be a commutative non-degenerate solution. Then, $\Ret(S,s)$ is an irretractable solution.
\end{prop}

\begin{proof}
    Let $x,y\in S$ such that $\bar{\theta}_{\bar{x}}=\bar{\theta}_{\bar{y}}$. Then
    \begin{eqnarray}
        \bar{\theta}_{\bar{x}}=\bar{\theta}_{\bar{y}} &\Longleftrightarrow&\theta_{\theta_x(z)}=\theta_{\theta_y(z)} \quad \forall z\in S \nonumber \\
         &\underset{\cref{eq3}}{\Longleftrightarrow}& \theta_{z}\theta_{x\cdotp z}^{-1}=\theta_{z}\theta_{y\cdotp z}^{-1} \quad \forall z\in S \nonumber \\
         &\underset{(S,s) \hspace{2mm} commutative}{\Longleftrightarrow}&
         \theta_{z}\theta_{x}^{-1}=\theta_{z}\theta_{y}^{-1} \quad \forall z\in S \nonumber \\
         &\Longleftrightarrow& \theta_x=\theta_y \nonumber \\
         &\Longleftrightarrow&\bar{x}=\bar{y} \nonumber
    \end{eqnarray}
Therefore the thesis follows.
\end{proof}

The next two results show that retractability of a solution $(S,s)$ can be detected by studying the action of $\mathcal{G}(S,s)$. 

\begin{theor}\label{carattirr}
     Let $(S,s)$ be a finite commutative non-degenerate solution. Then, $(S,s)$ is irretractable if and only if $\mathcal{G}(S,s)$ acts regularly on $S$.
\end{theor}

\begin{proof}
    Suppose that $(S,s)$ is irretractable. If $\mathcal{G}(S,s)$ does not act regularly on $S$, let $S_1$ and $S_2$ be distinct orbits of $S$. Then, by \cref{lemimp1} there exist $x_1\in S_1$ and $x_2\in S_2$ such that $\theta_{x_1}=\theta_{x_2}=id_S$, a contradiction.\\
    Conversely, suppose that $\mathcal{G}(S,s)$ acts regularly on $S$. If $\theta_x=\theta_y$, let $g\in \mathcal{G}(S,s)$ such that $g(x)=y$. Then, by \cref{lemimp11} we have that $\theta_y=\theta_{g(x)}=\theta_x g^{-1}$ hence $g^{-1}=id_S$ and the thesis follows.
\end{proof}




Now, we present a construction of irretractable commutative non-degenerate solutions. An arbitrary (finite or infinite) group is all we need. Moreover, we show that by this construction isomorphic solutions arise from isomorphic group.

\begin{prop}\label{esfond}
     Let $G$ be a group and $s_G$ be the map from $G$ to itself given by $s_G(g,h):=(g,g^{-1}h)$ for all $g,h\in G$. Then, $(G,s_G)$ is a bijective irretractable commutative non-degenerate solution and $\mathcal{G}(G,s_G)\cong G$.\\ Moreover, two groups $G$ and $G'$ are isomorphic if and only if $(G,s_G)$ and $(G',s_{G'})$ are isomorphic.
\end{prop}

\begin{proof}
    By a standard calculation, one can show that the map $s_G$ is bijective. To show that $(G,s_G)$ is a commutative solution we only have to prove that \cref{eq3} follows. Indeed, if $g,h,i$ are elements of  $G$, then 
    $$\theta_{\theta_g(h)}\theta_{g}(i)=(g^{-1}h)^{-1}g^{-1}i=h^{-1}i=\theta_h(i) $$
    and hence $(G,s_G)$ is a solution. Since the left multiplication by an element $g$ of $G$ is bijective, we obtain that $(G,s_G)$ is non-degenerate. Finally 
    $$ \theta_g=\theta_h \Longleftrightarrow g^{-1}z=h^{-1}z \quad \forall z\in G \Longleftrightarrow g=h $$
    hence $(G,s_G)$ is irretractable. Clearly, the subgroup generated by the maps $\theta_x$ is the one generated by the left multiplications $t_g:G\longrightarrow G$, $h\mapsto gh$, therefore $\mathcal{G}(G,s_G)\cong G$.\\
    If $G$ and $G'$ are isomorphic then $(G,s_G)$ and $(G',s_{G'})$ are clearly isomorphic. 
    Conversely, if we suppose that $G$ and $G'$ are groups such that $(G,s_G)$ and $(G',s_{G'})$ are isomorphic, let $\phi$ be a bijective map from $G$ to $G'$ such that $(\phi \times \phi) s_G= s_{G'} (\phi \times \phi ) $. This implies that $\phi(g^{-1} h)=\phi(g)^{-1}\phi(h)$ for all $g,h\in G$. Therefore if we set $g=1$ we obtain $\phi(h)=\phi(1)^{-1}\phi(h)$ for all $h\in G$ and hence $\phi(1)=1$, moreover if we set $h=1$ we obtain $\phi(g^{-1})=\phi(g)^{-1}$ for all $g\in G$ and hence    $$\phi(gh)=\phi((g^{-1})^{-1}h)=\phi(g^{-1})^{-1}\phi(h)=\phi((g^{-1})^{-1})\phi(h)=\phi(g)\phi(h)$$
    for all $g,h\in G$, hence the thesis follows.
\end{proof}

\noindent In the following results we show that all the finite irretractable solutions arise from the construction of the previous proposition.

\begin{theor}\label{classif}
    Let $(S,s)$ be an irretractable finite commutative non-degenerate solution. Then, $  (S,s)\cong (\mathcal{G}(S,s),s_{\mathcal{G}(S,s)})$.
\end{theor}

\begin{proof}
    Let $ (\mathcal{G}',\circ)$ be the opposite group of $\mathcal{G}(S,s)$. Let $z$ be the element of $S$ such that $\theta_z=id_S$ and $\phi$ be the function from $S$ to $\mathcal{G}'$ given by $\phi(z)=id_S$ and $\phi(g(z)):=g^{-1}$ for all $g\in \mathcal{G}'$. Then
    \begin{eqnarray}
        (\phi\times \phi) s(g(z),h(z))&=&(\phi\times \phi) (g(z),\theta_{g(z)}(h(z))) \nonumber \\
        &\underset{lemma}{=}& (\phi\times \phi) (g(z),g^{-1}(h(z))) \nonumber \\ 
        &=& (g^{-1},(g^{-1}h)^{-1}) \nonumber \\
        &=& (g^{-1},h^{-1}(g^{-1})^{-1}) \nonumber \\
        &=& (g^{-1},(g^{-1})^{-1}\circ  h^{-1}) \nonumber \\
        &=& s_{\mathcal{G}'}(g^{-1},h^{-1}) \nonumber \\
        &=& s_{\mathcal{G}'} (\phi \times \phi)(g(z),h(z)) \nonumber 
    \end{eqnarray}
    for all $g\in \mathcal{G}'$, then $(S,s)\cong (\mathcal{G}',s_{\mathcal{G}'})$. Since by \cref{esfond} $(\mathcal{G}',s_{\mathcal{G}'})\cong (\mathcal{G}(S,s),s_{\mathcal{G}(S,s)})$, the thesis follows.
\end{proof}



\begin{cor}\label{coroclass}
    Let $\mathcal{G}$ be a finite group. Then, $(\mathcal{G},s_{\mathcal{G}})$ is the unique irretractable finite commutative non-degenerate solution with permutation group isomorphic to $\mathcal{G}$.
\end{cor}

\begin{proof}
    It follows by \cref{classif}.
\end{proof}

\noindent We close the section recovering \cite[Proposition 4.5]{colazzo2020set} in the finite case.

\begin{cor}
    If the solution $(G,s_G)$ constructed in \cref{esfond} is an elementary abelian $2$-group, then $s_G$ is an involutive solution.\\
    Conversely, every finite irretractable involutive solution $(S,s)$ can be constructed by \cref{esfond}, using an elementary abelian $2$-group. Moreover, if $(S,s)$ and $(S',s')$ are two finite irretractable involutive solutions, then $(S,s)\cong (S',s')$ if and only if $|S|=|S'|$.
\end{cor}

\begin{proof}
    If $(G,s_G)$ is such that $G$ is an elementary abelian $2-$group, one can easily show that $s_G$ is an involutive solution.\\
    Conversely, if $(S,s)$ is a finite irretractable involutive solution, then it is commutative and non-degenerate by \cite[Corollaries 3.3 and 3.4]{colazzo2020set} and all the elements of the group $\mathcal{G}(S,s)$ have order $2$, hence $\mathcal{G}(S,s)$ is an elementary abelian $2-$group and by \cref{classif} $(S,s)$ is isomorphic to $(\mathcal{G}(S,s),s_{\mathcal{G}(S,s)})$. Finally, since two finite elementary abelian $2$-group are isomorphic if and only if they have the same cardinality, the last part of the statement follows by \cref{coroclass}.
\end{proof}

\section{Non-degenerate solutions on left-zero semigroup}

Following \cite[Problem 7]{mazzotta2023survey}, in this section we focus on solutions on left-zero semigroup. At first, we show the relation between retraction and subsolutions. In the core of the section, an extension-tool useful to construct \emph{all} the commutative non-degenerate solutions on left-zero semigroup is provided, giving a partial extension of the result obtained in \cite[Proposition 5.3]{colazzo2020set}. Finally, an explicit classification is given for the solutions having cyclic associated permutation group. 

\medskip

Recall that every solution on a left-zero semigroup is always commutative, therefore the adjective "commutative" could be dropped. At first, we show that for solution on left-zero semigroups the retraction always is isomorphic to a subsolution given by an arbitrary orbit.

\begin{prop}
    Let $(S,s)$ be a finite commutative non-degenerate solution on a left-zero semigroup $(S,\cdotp)$ and let $Z$ be an orbit of $S$ respect to the action of $\mathcal{G}(S,s)$. Then, the solution $\Ret(S,s)$ is isomorphic to $(Z,s_{|_{Z\times Z}})$.
\end{prop}

\begin{proof}
    Let $\phi$ be the map from $Z$ to $\Ret(S,s) $ given by $\phi(z):=\bar{z}$ for all $z\in Z$. By \cref{lemgr} the map $\phi$ is surjective and by \cref{cardretr}, $Z$ and  $\Ret(S,s) $ have the same size, then $\phi$ is bijective. Now, if $x,y\in Z$ we have
    \begin{eqnarray}
        (\phi\times \phi) s_{|_{Z\times Z}}(x,y) &=& (\phi\times \phi) (x,\theta_x(y))\nonumber \\
        &=& (\bar{x},\overline{\theta_{x}(y)})\nonumber \\
        &=& (\bar{x},\overline{\theta}_{\bar{x}}(\bar{y}))\nonumber \\
        &=& \bar{s}(\phi\times \phi)(x,y)\nonumber 
    \end{eqnarray}
    therefore the thesis follows.
\end{proof}

Inspired by dynamical extensions of set-theoretic solutions of the Yang-Baxter equation (see \cite{vendramin2016extensions}), in the following theorem we develop a method to construct all the finite commutative non-degenerate solutions on left-zero semigroups.

\begin{theor}\label{costruzlz}
    Let $G$ be a group, $X$ a set and $\pi:G\times G\longrightarrow Sym(X)$, $\pi(a,b):=\pi_{a,b}$ a map such that $\pi_{a^{-1} b,a^{-1} c}\pi_{a,c}=\pi_{b,c} $ for all $a,b,c\in G$. Then, the pair $(X\times G,s)$ given by $s((x,a),(y,b)):=((x,a),(\pi_{a,b}(y),a^{-1} b))$ is a bijective commutative non-degenerate solution on the left-zero semigroup having $X\times G$ as underlying set. Moreover, $Ret(X\times G,s)\cong (G,s_G)$.\\
    Conversely, every finite commutative non-degenerate solution on a left-zero semigroup can be constructed in this way.
\end{theor}

\begin{proof}
    Let $G,X$, $\pi$ and $s$ as in the statement and let $(x,a)$ $(y,b)$ be arbitrary elements of $X\times G$. By an easy calculation one can show that the map $\theta_{(x,a)}$ given by $\theta_{(x,a)}(z,c):=(\pi_{a,c}(z),a^{-1} c)$ is bijective. Moreover, since $\pi_{a,b}\in Sym(X)$ for all $a,b\in G$, the bijectivity of $s$ also follows. Therefore to show the first part of the statement we have to prove that the equality  $\theta_{\theta_{(x,a)}(y,b)}\theta_{(x,a)}=\theta_{(y,b)}. $ Then, we have 
    \begin{eqnarray}
    \theta_{\theta_{(x,a)}(y,b)}\theta_{(x,a)}(z,c) &=& \theta_{(\pi_{a,b}(y),a^{-1} b)}(\pi_{a,c}(z),a^{-1} c) \nonumber \\
    &=& (\pi_{a^{-1} b,a^{-1} c}(\pi_{a,c}(z)),(a^{-1} b)^{-1} a^{-1} c) \nonumber \\
    &=& (\pi_{b,c}(z),b^{-1} c) \nonumber \\
    &=& \theta_{(y,b)}(z,c) \nonumber 
    \end{eqnarray}
    for all $(z,c)\in X\times G$, hence the pair $(X\times G,s)$ is a bijective commutative non-degenerate solution on the left-zero semigroup having $X\times G$ as underlying set. Now, since $\theta_{(x,a)}=\theta_{(x',a')}$ if and only if $a=a'$, it is well-defined the map $\psi$ from $Ret(X\times G,s) $ to $(G,s_G)$ by $\psi(\overline{(x,a)}):=a$ for all $(x,a)\in X\times G$. By a standard calculation, we obtain that $\psi$ is an isomorphism.\\
    Conversely, suppose that $(S,s)$ is a finite commutative non-degenerate solution on a left-zero semigroup and set $G:=\mathcal{G}(Ret(S,s))$. Then, by \cref{cardretr}, \cref{retrirr} and \cref{classif} $S$ can be identified with the set $X\times G$ for a suitable set $X$, where $\theta_{(x,a)}=\theta_{(y,b)}$ if and only if $a=b$. Under this identification, there exist a function $q:G\times G\times X\longrightarrow Sym(X)$, $q(a,b,x):=q_{a,b,x}$ such that $s$ can be written as 
    $$s((x,a),(y,b))=((x,a),(q_{a,b,x}(y),a^{-1} b) $$
    for all $(x,a),(y,b)\in X\times G$. Since $\theta_{(x,a)}=\theta_{(y,b)}$ if and only if $a=b$, we obtain $q_{a,b,x}=q_{a,b,y}$ for all $a,b\in G$ and $x,y\in X$, therefore, if $x$ is an arbitrary element of $X$, we can set $\pi(a,b):=q_{a,b,x}$ for all $a,b\in G$. Finally, since $(X\times G,s)$ is a solution, we have that $\pi_{a^{-1} b,a^{-1} c}\pi_{a,c}=\pi_{b,c} $ for all $a,b,c\in G$ and hence the second part of the statement follows.
 \end{proof}

\begin{rem}
     Actually, we do not know if the "converse part" of the previous theorem holds for infinite solutions.
\end{rem} 

\noindent By \cref{costruzlz}, we are able to obtain a family of solutions that includes the one constructed in \cite[Proposition 5.1]{colazzo2020set}

\begin{prop}
     Let $G$ be a group, $X$ a set and $\sigma:G\longrightarrow Sym(X)$, $\sigma(a):=\sigma_a$ a map from $G$ to $Sym(X)$. Then, the pair $(X\times G,s)$ given by 
     $$s((x,a),(y,b)):=((x,a),(\sigma_{a^{-1}\circ b}\sigma_b^{-1}(y),a^{-1}\circ b))$$
     is a commutative non-degenerate solution on the left-zero semigroup having $X\times G$ as underlying set. This solution is involutive if and only if $G$ is an elementary abelian $2-$group. Moreover, $Ret(X\times G,s)\cong (G,s_G)$.    
\end{prop}

\begin{proof}
    Let $\pi$ the map from $G\times G $ to $Sym(X)$ given by $\pi_{a,b}:=\sigma_{a^{-1} b}\sigma_b^{-1}$ for all $a,b\in G$. Then
    \begin{eqnarray}
        \pi_{a^{-1} b,a^{-1} c}\pi_{a,c}(z) &=& \pi_{a^{-1} b,a^{-1} c}\sigma_{a^{-1} c}\sigma_c^{-1}(z) \nonumber \\
        &=& \sigma_{b^{-1} a a^{-1} c}\sigma_{a^{-1} c}^{-1}\sigma_{a^{-1} c}\sigma_c^{-1}(z) \nonumber \\
        &=& \sigma_{b^{-1} c}\sigma_{c}^{-1}(z) \nonumber \\
        &=& \pi_{b,c}(z)
    \end{eqnarray}
    for all $a,b,c\in G$, therefore the first part of the statement follows by \cref{costruzlz}. By the same theorem, we have that $Ret(X\times G,s)$ is isomorphic to $ (G,s_G)$.\\
    Finally, if the solution is involutive, then so is $(G,s_G)$, hence $G$ must be an elementary abelian $2-$group. Conversely, if $G$ is an elementary abelian $2-$group, by \cite[Proposition 5.1]{colazzo2020set} we have that $(X\times G,s)$ is involutive, hence the statement follows.
\end{proof}

Even if we restict to non-degenerate solutions on left-zero semigroups, providing a quite explicit classification seems to be hard. However, in the following theorem we show that this is possible if we consider solution on left-zero semigroups with cyclic permutation groups. 

\begin{theor}\label{classcic}
    Every commutative non-degenerate solution on an arbitrary left-zero semigroup of size $r$ with cyclic associated permutation group is completely determined by two positive integers $m$ and $n$ such that $mn=r$, where $m$ is the number of the orbits respect to the action of the permutation group and $n$ is the size of every orbit.
\end{theor}

\begin{proof}
    Let $m,n$ be positive integers such that $mn$ is equal to $r$ and $S:=\{x_{i,j}\}_{i\in \{1,...,m\},j\in \{1,...,n \}}$ be a the set having size $r$. Take $\alpha$ the permutation given by $\alpha(x_{i,j}):=x_{i,j+1}$ for all $i,j$, and $\theta$ the map from $S$ to  $<\alpha>$ given by $\theta_{x_{i,j}}:=\alpha^{2-j} $ for all $i\in \{1,...,m\},j\in \{1,...,n\}$. By a standard calculation, one can show that $S$ and $\theta$ give rise to a commutative non-degenerate solution $(S,s)$ on a left-zero semigroup on $S$. Moreover, $\mathcal{G}(S,s)$ is isomorphic to the subgroup generated by $\alpha$. To prove the statement we have to show  that every solution with cyclic permutation group and $m$ orbits of size $n$ is isomorphic to $(S,s)$.
    Now, let $(S',s')$ be a commutative non-degenerate solution having size $r$ and such that the permutation group is a cyclic group and has $m$ orbits of size $n$. Then, after renaming the variables, we can suppose that $S':=\{y_{i,j}\}_{i\in \{1,...,m\},j\in \{1,...,n \}}$ and $\mathcal{G}(S',s')$ is generated by the permutation $\alpha'$ given by $\alpha'(y_{i,j}):=y_{i,j+1}$ for all $i,j$. Moreover, by \cref{lemgr}, after renaming again the variables, we can suppose that $\theta_{y_{i,1}}=\alpha'$ for all $i\in \{1,....,m\}$. By a standard calculation, we obtain that $\theta_{y_{i,j}}:=\alpha'^{2-j} $ and the map $\phi$ from $S$ to $S'$ given by $\phi(x_{i,j}):=y_{i,j}$ for all $i\in \{1,...,m\}$ and $j\in \{1,...,n\}$ is an isomorphism from $(S,s)$ to $(S',s')$, therefore the claim follows.
\end{proof}

\section{Solutions on left-zero semigroups having small size}

As a consequences of the results obtained in the previous sections, we provide some classification results on solutions having small size. 

\begin{prop}\label{propoleft}
    There exist $2$ commutative non-degenerate solutions on a left-zero semigroup of prime size $p$: 
    \begin{itemize}
        \item[1)] the solution $(\mathbb{Z}/p\mathbb{Z},s)$ given by $s(x,y)=(x,-x+y)$ for all $x,y\in \mathbb{Z}/p\mathbb{Z}$;
        \item[2)] the solution $(\mathbb{Z}/p\mathbb{Z},s)$ given by $s(x,y)=(x,y)$ for all $x,y\in \mathbb{Z}/p\mathbb{Z}$.
    \end{itemize}
    
\end{prop}

\begin{proof}
    Since by \cref{teo1} the action of the permutation group of a commutative non-degenerate solution is semi-regular, if $(S,s)$ is a commutative non-degenerate solutions with size $p$ we must have one of the two following condition:
    \begin{itemize}
        \item[-] a unique orbit of size $p$, hence $|\mathcal{G}(S,s)|=p$ and by \cref{coroclass} $(S,s)$ is isomorphic to the solution 1);
        \item[-] $p$ orbits of size $1$, hence $(S,s)$ must be the solution 2).
    \end{itemize}     
\end{proof}

\begin{cor}
    For every prime number $p$, there exist a unique irretractable commutative non-degenerate solution $(S,s)$ of size $p$ and it is given by $S:=\mathbb{Z}/p\mathbb{Z}$ and $s(x,y):=(x,-x+y)$ for all $x,y\in S$.
\end{cor}

\begin{proof}
    Since the underlying semigroup of an irretractable solution must be a left-zero semigroup, the thesis follows by \cref{propoleft}.
\end{proof}

\begin{prop}\label{propoleftpq}
    Let $p,q$ be distinct prime numbers with $p>q$. A complete list of the commutative non-degenerate solutions of size $pq$ on a left-zero semigroup is the following: 
    \begin{itemize}
     \item[1)] the solution $(\mathbb{Z}/pq\mathbb{Z},s)$ given by $s(x,y)=(x,y)$ for all $x,y\in \mathbb{Z}/pq\mathbb{Z}$;
    \item[2)] the solution $(\mathbb{Z}/pq\mathbb{Z},s)$ given by $s(x,y)=(x,-x+y)$ for all $x,y\in \mathbb{Z}/pq\mathbb{Z}$;
    \item[3)] if $p-1\equiv 0$ $(mod \hspace{2mm} q)$, the solution $(\mathbb{Z}/p\mathbb{Z}\rtimes \mathbb{Z}/q\mathbb{Z},s)$ given by $s(x,y)=(x,x^{-1}\circ y)$ for all $x,y\in \mathbb{Z}/p\mathbb{Z}\rtimes \mathbb{Z}/q\mathbb{Z}$;
    \item[4)] the unique solution of \cref{classcic} with cyclic permutation group and $p$ orbits of $q$ elements;
    \item[5)] the unique solution of \cref{classcic} with cyclic permutation group and $q$ orbits of $p$ elements.
  \end{itemize}    
\end{prop}

\begin{proof}
    Suppose that $(S,s)$ is a commutative non-degenerate solution of size $pq$ on a left-zero semigroup. Since by \cref{teo1} $\mathcal{G}(S,s)$ is semi-regular, we have that $|\mathcal{G}(S,s)|\in \{1,p,q,pq\}$. If $(S,s)$ is irretractable, then by \cref{carattirr} we must have $|\mathcal{G}(S,s)|=pq$ and hence by \cref{coroclass} we have the solution $2)$ or, if $p-1\equiv 0$ $(mod \hspace{2mm} q)$, we can have solution $3)$. Finally, if $(S,s) $ is retractable we have $|\mathcal{G}(S,s)|\in \{1,p,q\}$ and hence $\mathcal{G}(S,s)$ is a cyclic group. Therefore, by \cref{classcic} we find solution $1)$, solution $4)$ or solution $5)$.
\end{proof}

\begin{cor}
    Let $p,q$ be distinct prime numbers with $p>q$. Then, the irretractable commutative non-degenerate solutions of size $pq$ are the following: 
    \begin{itemize}
   \item[1)] the solution $(\mathbb{Z}/pq\mathbb{Z},s)$ given by $s(x,y)=(x,-x+y)$ for all $x,y\in \mathbb{Z}/pq\mathbb{Z}$;
    \item[2)] if $p-1\equiv 0$ $(mod \hspace{2mm} q)$, the solution $(\mathbb{Z}/p\mathbb{Z}\rtimes \mathbb{Z}/q\mathbb{Z},s)$ given by $s(x,y)=(x,x^{-1}\circ y)$ for all $x,y\in \mathbb{Z}/p\mathbb{Z}\rtimes \mathbb{Z}/q\mathbb{Z}$.
  \end{itemize}    
\end{cor}

\begin{proof}
Since the underlying semigroup of an irretractable solution must be a left-zero semigroup, the thesis follows by \cref{carattirr} and \cref{propoleftpq}.
\end{proof}

The proofs of the following results, that classify commutative non-degenerate solutions of size $p^2$, are omitted since they are similar to the ones of the previous results.

\begin{prop}\label{propoleftp2}
    Let $p$ be a prime number. A complete list of the commutative non-degenerate solutions of size $p^2$ on a left-zero semigroup is the following: 
    \begin{itemize}
     \item[1)] the solution $(\mathbb{Z}/p^2\mathbb{Z},s)$ given by $s(x,y)=(x,y)$ for all $x,y\in \mathbb{Z}/pq\mathbb{Z}$;
    \item[2)] the solution $(\mathbb{Z}/p^2\mathbb{Z},s)$ given by $s(x,y)=(x,-x+y)$ for all $x,y\in \mathbb{Z}/pq\mathbb{Z}$;
    \item[3)]  the solution $(\mathbb{Z}/p\mathbb{Z}\times \mathbb{Z}/p\mathbb{Z},s)$ given by $s(x,y)=(x,-x+y)$ for all $x,y\in \mathbb{Z}/p\mathbb{Z}\times \mathbb{Z}/p\mathbb{Z}$;
    \item[4)] the unique solution of \cref{classcic} with cyclic permutation group and $p$ orbits of $p$ elements.
  \end{itemize}    
\end{prop}

\begin{cor}
    Let $p$ be a prime number. Then, the irretractable commutative non-degenerate solutions of size $p^2$ are the following: 
    \begin{itemize}
   \item[1)] the solution $(\mathbb{Z}/p^2\mathbb{Z},s)$ given by $s(x,y)=(x,-x+y)$ for all $x,y\in \mathbb{Z}/p^2\mathbb{Z}$;
    \item[2)] the solution $(\mathbb{Z}/p\mathbb{Z}\times \mathbb{Z}/p\mathbb{Z},s)$ given by $s(x,y)=(x,-x+y)$ for all $x,y\in \mathbb{Z}/p\mathbb{Z}\times \mathbb{Z}/p\mathbb{Z}$.
  \end{itemize}    
\end{cor}




\bibliographystyle{elsart-num-sort}
\bibliography{Bibliography2}

\end{document}